\documentclass[11pt]{amsart}

\usepackage{amsthm}
\usepackage{amsmath}
\usepackage{amssymb}
\usepackage{color}
\usepackage[toc,page]{appendix}
\usepackage{enumerate}
\usepackage{filecontents}

\newtheorem{thm}{Theorem}[section]
\newtheorem{theorem}[thm]{Theorem}

\newtheorem{cor}[thm]{Corollary}

\newtheorem{corollary}[thm]{Corollary}

\newtheorem{observation}[thm]{Observation}

\newtheorem*{theorem*}{Theorem}
\newtheorem*{lemma*}{Lemma}

\newtheorem{lemma}[thm]{Lemma}
\newtheorem{example}{Example}
\newtheorem{proposition}[thm]{Proposition}

\newtheorem{definition}[thm]{Definition}
\theoremstyle{remark}
\newtheorem{remark}[thm]{Remark}

\newcommand{\RR}{\mathbb R}

\newcommand{\HH}{\mathbb H}

\newcommand{\ip}[2]{\left\langle#1,#2\right\rangle}
\newcommand{\abs}[1]{\left| #1 \right|}

\newcommand{\vertiii}[1]{{\left\vert\kern-0.25ex\left\vert\kern-0.25ex\left\vert #1 
    \right\vert\kern-0.25ex\right\vert\kern-0.25ex\right\vert}}
\allowdisplaybreaks

\begin{document}

\nocite{*}

\title[Outer Product Frames]{Riesz Outer Product 
Hilbert space Frames:\\ Quantitative Bounds, Topological Properties,\\ and Full Geometric Characterization}

\author[Casazza, Pinkham, Tuomanen
 ]{Peter G. Casazza, Eric Pinkham and Brian Tuomanen}
\address{Department of Mathematics, University
of Missouri, Columbia, MO 65211-4100}

\thanks{The authors were supported by
NSF DMS 1307685; NSF ATD 1042701 and 1321779; AFOSR DGE51: FA9550-11-1-0245
}

\email{casazzap@missouri.edu, eap9qc@missouri.edu, bpt6gc@missouri.edu}

\begin{abstract}
Outer product frames are important objects in Hilbert space
frame theory.  But very little is known about them.
In this paper, we make the first detailed study of the family of outer product frames induced directly by vector sequences. We are interested in both the quantitative attributes of these outer product sequences (in particular, their Riesz and frame bounds), as well as their independence and spanning properties. We show that Riesz sequences of vectors yield Riesz sequences of outer products with the same (or better) Riesz bounds. Equiangular tight frames are shown to produce Riesz sequences with optimal Riesz bounds for outer products. We provide constructions of frames which produce Riesz outer product bases with ``good'' Riesz bounds. We show that the family of unit norm frames which yield independent outer product sequences is open and dense (in a Euclidean-analytic sense)  within the topological space  $ \otimes_{i=1}^M  S_{N - 1}$ where $M$ is less than or equal to the dimension of the space of symmetric operators on $\HH^N$;  that is to say, almost every frame with such a bound on its cardinality will induce a set of independent outer products.  Thus, this would mean that finding the necessary and sufficient conditions such that the induced outer products are dependent is a more interesting question. For the coup de gr\^{a}ce, we give a full analytic and geometric classification of such sequences which produce dependent outer products. 
\end{abstract}

\maketitle
\pagebreak
\tableofcontents
\pagebreak
\section{Introduction}
	In this paper we are concerned with two classes of sequences for finite dimensional Hilbert spaces; frames and Riesz sequences. The first has its origins in Harmonic analysis and was first introduced in 1952 by Duffin and Schaefer in \cite{DS}. Frames provide redundant representations for vectors in a Hilbert space. This inherent property allows for the representation of any element of a Hilbert space in infinitely many ways. This gives natural robustness to noise \cite{mallat_book} and erasures \cite{1402167}. Riesz sequences have been around even longer though perhaps not as thoroughly studied. A Riesz basis provides a basis for a Hilbert space with quantitative bounds on the norm of a vectors representation in terms of its coefficients. We will be relating these two classes of sequences through \emph{outer products}. Outer products  can be abstractly considered as tensors or in our case more frequently as rank one projections. Outer products have recently appeared in numerous papers (for instance, \cite{scalable_frames_convex_geometry,scalable_note}) regarding the \emph{scaling problem}. Here we give 
	the first  thorough study of frames and Riesz sequences of outer products.
    
We start by introducing some of the basic terminology used throughout this paper. Though most of the necessary material is provided here, we assume that the reader has a familiarity with the basics of frame theory. The reader may wish to review \cite{frames_for_undergraduates, MR2422239, MR1946982,tropp,casazza2012finite}.

We assume that all vectors are column vectors.

\begin{definition}
        A sequence of vectors $\{\phi_i\}_{i=1}^M\subset\HH^N$ is a \emph{frame} for $\HH^N$ provided there exists $0<A\leq B<\infty$ such that
        \[A\|\psi\|^2\leq \sum_{i=1}^M|\ip{\phi_i}{\psi}|^2\leq B\|\psi\|^2\]
        for all $\psi\in \HH^N$. $A$ and $B$ are called the \emph{lower} and \emph{upper} frame bouns respectively.
\end{definition}
In the finite dimensional setting, a frame is just a spanning set, see \cite{frames_for_undergraduates}. It should be noted, that there are many frame bounds for a given frame. The largest lower frame bound and the smallest upper frame bound are the \emph{optimal} frame bounds. We characterize several classes of frames of particular interest by their frame bounds. If $A=B$ the frame is said to be a \emph{tight frame}, and if $A=B=1$ it is a \emph{Parseval frame}. These classes are particularly useful for reasons we will see below.

There are several important operators which go along with the study of frames. For the most part we will not be needing these but for completeness we include them.

\begin{definition}
        Let $\Phi=\{\phi_i\}_{i=1}^M$ be a frame for $\HH^N$.
                \begin{enumerate}
                        \item The synthesis operator of 
                        $\Phi$ is 
                        \[T:\ell_2^M\to \HH^N \ \ \ \ T:(a_i)_{i=1}^M\mapsto \sum_{i=1}^M a_i\phi_i.\]
                        Its matrix representation is
                        \[T=\begin{bmatrix}
                        | & | & & | \\
                        \phi_1 & \phi_2 & \cdots & \phi_M\\ 
                        | & | & & |
                        \end{bmatrix}.\]
                
                        \item The \emph{analysis operator} of $\Phi$ is the Hermitian adjoint of $T$,
                        \[T^*:\HH^N\to \ell_2^M \ \ \ \ T^*:\psi \mapsto (\ip{\psi}{\phi_i})_{i=1}^M.\]
                        \item The \emph{frame operator} 
                        of $\Phi$ is $S=TT^*$ so that
                        \[S:\HH^N \to \HH^M \ \ \ \ S:\psi\mapsto \sum_{i=1}^M\ip{\psi}{\phi_i}\phi_i.\]
                        
                        \item The \emph{Gram matrix} of 
 $\Phi$ is \[ G(\Phi) = T^*T = [\langle \phi_i,\phi_j]
 _{i,j=1}^M.\]
                \end{enumerate}
                It follows that the non-zero eigenvalues of
                $S$ and $G(\Phi)$ are equal and so the largest smallest non-zero eigenvalues of $G(\Phi)$ are the lower and upper frame bounds of $\Phi$.
\end{definition} 
The frame operator exhibits great utility in understanding frame properties.
\begin{theorem}
        Let $\{\phi_i\}_{i=1}^M$ be a frame for $\HH^N$. Then the frame operator $S$ is self-adjoint, positive, and invertible. Furthermore, the largest and smallest eigenvalues of $S$ are precisely the optimal upper and lower frame bounds of $\{\phi_i\}_{i=1}^M$ respectively.
\end{theorem}
{\it Reconstruction} is carried out by
\[\psi=SS^{-1}\psi=\sum_{i=1}^M\ip{\psi}{\phi_i}S^{-1}\phi_i=\sum_{i=1}^M\ip{\psi}{S^{-1}\phi_i}\phi_i.\]
This provides useful representations of any vector in our Hilbert space through the frame operator. For applications, we want the frame operator to be as well conditioned as possible for stability of the representation. This means that frames which are close to being tight are more desirable than those with arbitrarily small lower frame bounds. Particularly useful frames for encoding and decoding as above are tight frames. Tight frames have the important property that their frame operator is a multiple of the identity and hence inverting them is trivial. This is especially useful when our space has very high dimension as is common in applications.

The second class of sequences we will be examining are \emph{Riesz sequences}. 
\begin{definition}
        A sequence of vectors $\{\phi_i\}_{i=1}^M\subset \HH^N$ is a \emph{Riesz sequence} provided there exists $0<A\leq B <\infty$ such that
        \[A\sum_{i=1}^M|a_i|^2\leq \left\|\sum_{i=1}^Ma_i\phi_i\right\|^2\leq B\sum_{i=1}^M|a_i|^2\]
        for all $(a_i)_{i=1}^M\in \HH^M$. $A$ and $B$ are called the \emph{lower} and \emph{upper} Riesz bounds respectively.
\end{definition}
Again when dealing with finite dimensional vector spaces, these objects have a very simple characterization: a set is Riesz if and only if it is linearly independent. We will use \emph{independent} and \emph{Riesz} nearly  interchangeably in this paper. We will use \emph{Riesz} when we are particularly concerned with not only the independence but also the Riesz bounds.

The final object we need to define before beginning our study of outer products of Riesz squences and frames is the \emph{outer product} of two vectors. 

\begin{definition}
        For $\phi, \psi\in \HH^N$, define the outer product of $\phi$ and $\psi$ by $\phi\psi^*$ in terms of standard matrix multiplication. 
        For any vector $\phi\in \HH^N$, we define the \emph{induced outer product of $\phi$} as $\phi \phi^*$.  Note that if $\phi$ is a unit norm vector, then this will be a rank one orthogonal projection.
\end{definition}
%
%

Much of the following work will be in the space of $N\times N$ matrices over the real or complex fields. We will denote these spaces as $\HH^{N\times N}$, and as needed clarifying the base field. In the case that we are restricting our attention to the symmetric or self-adjoint matrices we will use $\mathrm{sym}(\HH^{N\times N})$. To simplify notation, given $S\in \HH^{N\times N}$ we will use $S^*$ for both the Hermitian adjoint and transpose understanding that the underlying field determines which is at play.
\begin{remark}
The ambient space of outer products is the space of self-adjoint matrices on $\HH^N$. It has dimension $N(N+1)/2$ if $\HH$ is real. If $\HH$ is complex, the space of self-adjoint matrices does not form a \emph{complex} vector space but instead a \emph{real} vector space, as such it has dimension $N^2$. 
\end{remark}
For $\phi,\psi\in \HH^N$ we will denote the $i^{th}$ entry of $\phi$ by $\phi(i)$. For a matrix $S$ we will denote the $(i,j)^{th}$ entry by $S[i,j]$. 

We will equip these vector spaces with the Frobenius matrix inner product. 
\begin{definition}
        Let $S,T\in \HH^{N\times N}$. The \emph{Frobenius inner product} is 
        \[\ip{S}{T}_F=\mathrm{Tr}(S^*T)=\mathrm{Tr}(ST^*)=\sum_{i=1}^N\sum_{j=1}^N S[i,j]T[i,j].\]
        We may drop the subscript $F$ when no confusion
        will arise.
\end{definition}
For given $\phi,\psi\in \HH^N$ we will use the usual $\ell_2$ inner product
\[\ip{\phi}{\psi}=\sum_{i=1}^N\phi(i)\overline{\psi(i)}.\]

Throughout this paper we will use $I_N$ to be the $N\times N$ identity matrix and $1_N\in \HH^N$ to be the vector $1_N$ to be the vector of all $1$'s.

\section{Some Basic Calculations}

The primary goal of this paper is investigating the independence of outer products of sequences of vectors. 
We start with a simple calculation.

\begin{lemma}\label{pc2}
For any vectors $\phi_1,\phi_2\in \HH^N$ we have
\[ \langle \phi_1\phi_1^*,\phi_2\phi_2^*\rangle_F
= |\langle \phi_1,\phi_2\rangle|^2.\]

\end{lemma}

\begin{proof}
We compute:
\begin{align*} \langle \phi_1\phi_1^*,\phi_2\phi_2^*\rangle_F
&= Tr(\phi_2\phi_2^* \phi_1\phi_1^*)\\
&= Tr(\phi_2\ip{\phi_2}{\phi_1}\phi_1^*)\\
&= Tr(\ip{\phi_1}{\phi_2}\ip{\phi_2}{\phi_1})\\
&= |\ip{\phi_1}{\phi_2}|^2.
\end{align*}
\end{proof}

\begin{cor}
$\phi_1 \perp \phi_2$ in $\HH^N$ if and only if  $\phi_1\phi_1^* \perp
\phi_2\phi_2^*$ in $\mathrm{sym}(\HH^{N\times N})$.
\end{cor}

\begin{proposition}\label{pc1}
        Let $\{\phi_i\}_{i=1}^M$ be a unit norm frame for $\HH^N$. The family $\{\phi_i\phi_i^*\}_{i=1}^M$ is linearly independent
        if and only if there are scalars $\{a_i\}_{i=1}^M$ with $a_i \ge 0$ and $I\subset [M]$ so that 
        if $S_I$ is the frame operator $\{\sqrt{a_i}\phi_i\}_{i\in I}$, and $S_{I^c}$ is the frame operator of the frame sequence $\{\sqrt{-a_i}\phi_i\}_{i\in I^c}$.  Then
        \[S_I=S_{I^c}.\]
\end{proposition}

\begin{proof}We observe that

\[ \sum_{i=1}^M a_i\phi_i\phi_i^*=0, \]
if and only if letting $I=\{1\le i \le M:a_i\ge 0\}$, we have
\[ \sum_{i\in I}a_i\phi_i\phi_i^*= \sum_{i\in I}(\sqrt{a_i}\phi_i)(
\sqrt{a_i}\phi_i)^* = S_I = S_{I^c} = \sum_{i\in I^c}
(\sqrt{-a_i}\phi_i)(\sqrt{-a_i}\phi_i)^*.\] 
\end{proof}

One of the main tools in examining the outer products of a collection of vectors will be the Gram matrices of our vectors. When dealing with a Riesz sequence, or a linearly independent collection of vectors, the Gram matrix will be positive-definite. Furthermore, the largest and smallest eigenvalues of this matrix represent the upper and lower Riesz bounds of our sequence respectively. In the case of redundant frames, the Gram matrix is singular. However, the largest and smallest non zero eigenvalues give the upper and lower frame bounds respectively. We will need the
 the Gram matrix matrix of out products.

\begin{theorem}
        Let $\{\phi_i\}_{i=1}^M$ be a sequence of vectors in $\HH^N$. Then the Gram matrix of $\{\phi_i\phi_i^*\}_{i=1}^M$ is
        \[G=[|\ip{\phi_i}{\phi_j}|^2].\]
        Moreover,
        \begin{enumerate}
                \item If $\{\phi_i\phi_i^*\}_{i=1}^M$ is a Riesz sequence, then the optimal Riesz bounds are the largest and smallest eigenvalue of $G$. 
                \item If $\{\phi_i\phi_i^*\}_{i=1}^M$ is a frame then the frame bounds are the largest and smallest non-zero eigenvalues of $G$.
        \end{enumerate}
\end{theorem}

The Gram matrix of the induced outer products can be represented in terms of the Gram matrix of the original vectors by using the Hadamard product.

\begin{definition}
        Given two matrices $A=[a_{ij}]$ and $B=[b_{ij}]$ in $\HH^{M\times N}$ the \emph{Hadamard product} of $A$ and $B$ is
        \[A\circ B=[a_{ij}b_{ij}].\]
\end{definition}

The following is a well known theorem about Hadamard products, see \cite{HJ_topics} for example. 

\begin{theorem}\label{gram}
        Let $A$ and $B$ be Hermitian with $A=[a_{ij}]$ positive semidefinite. Any eigenvalue $\lambda(A\circ B)$ of $A\circ B$ satisfies
        \begin{align*}
        \lambda_{min}(A)\lambda_{min}(B)&\leq [\min_{i}a_{ii}]\lambda_{min}(B)\\
        &\leq \lambda(A\circ B)\\
        &\leq [\max_{i}a_{ii}]\lambda_{max}(B)\\
        &\leq \lambda_{max}(A)\lambda_{max}(B).
        \end{align*}
\end{theorem}

\begin{corollary}\label{riesz_riesz}
        If $\{\phi_i\}_{i=1}^M$ is a unit norm Riesz 
        sequence with Riesz bounds $A$ and $B$ then $\{\phi_i\phi_i^*\}_{i=1}^M$ is also Riesz with the same Riesz bounds.
\end{corollary}
\begin{proof}
        Let $G$ be the Gram matrix of $\{\phi_i\}_{i=1}^M$ and $H$ be the Gram matrix of the induced outer products. Then 
        \[H=G\circ \overline{G}=G\circ G^T.\]
        Since $G$ and $G^T$ have the same eigenvalues and the diagonal entries of $G$ are $\|\phi_i\|^2$, the result follows.
\end{proof}
The above proofs are convenient for their conciseness but mask much of the machinery at use. For a direct proof which may be more enlightening see Appendix~\ref{brians_argument}.

It may not be surprising that unit norm
Riesz sequences produce Riesz outer products--but what is surprising is that \emph{the same Riesz bounds hold!} That is, Riesz bounds cannot worsen when moving to the outer product space. A natural question to ask at this point is whether the Riesz bounds of the induced outer products can be better than the Riesz bounds of the original vectors.  The answer is {\it 
yes}.
\begin{example}
        Let $\phi_1=[0,1]^T,$ $\phi_2=[\sqrt{\varepsilon},\sqrt{1-\varepsilon}]^T$ for $0< \varepsilon<1$. Then $\{\phi_i\}_{i=1}^2$ is Riesz with Riesz bounds $1-\sqrt{\varepsilon }$ and $1+\sqrt{\varepsilon}$ while $\{\phi_i\phi_i^*\}_{i=1}^2$ is Riesz with bounds $1-\varepsilon$ and $1+\varepsilon$. 
\end{example}
\begin{proof}
        The Gram matrix of $\{\phi_1,\phi_2\}$ is
        \[\begin{bmatrix}
        1 & \sqrt{\varepsilon} \\
        \sqrt{\varepsilon} & 1 
        \end{bmatrix}\]
        while that of $\{\phi_1\phi_1^*,\phi_2\phi_2^*\}$ is
        \[\begin{bmatrix}
        1&\varepsilon\\
        \varepsilon & 1
        \end{bmatrix}.\] The eigenvalues of these matrices are as required.
\end{proof}

\section{Some Results Guaranteeing Riesz Outer Products}
The preceding sections show the difficulty in deciding whether a dependent collection of vectors produces independent outer products. Later, we will see that ``most'' 
family of vectors
induce independent outer product sequences. A more relevant question is ``which frames induce dependent outer products?'' We will see a full characterization of all frames which induce dependent outer products. For now, we give a few simple observations which can be used to quickly check whether a sequence will produce independent outer products. 

\subsection{Sparsity and Vectorized Outer Products}

\begin{definition}
        Let $\phi\in \HH^N$. Define the \emph{vectorization} of $\phi\phi^*$ as the vector obtained by stacking the columns on top of each other. That is, the vectorization of $\phi\phi^*$ is 
        \[\begin{bmatrix}
        \phi(1)\overline{\phi}\\
        \phi(2)\overline{\phi}\\
        \vdots \\
        \phi(N)\overline{\phi}
        \end{bmatrix}\]
        where $\phi(k)$ is the $k$th entry of $\phi$.
\end{definition}
\begin{proposition}\label{sparsity}
        Let $\{\phi_i\}_{i=1}^M$ be a frame for $\HH^N$ with no zero vectors. For $k=1,\dots, N$ define $I_k=\{i:\phi_i(k)\neq 0\}$. If $\{\phi_i\}_{i\in I_k}$ is independent for all $k$, then $\{\phi_i\phi_i^*\}_{i=1}^M$ is independent.

\end{proposition}
\begin{proof}
        Let $\{\phi_{i}\}_{i=1}^M$ be a frame with the properties as stated. Let $C_i$ be the vectorization of $\phi_i\phi_i^*$. Now consider the synthesis operator of $\{C_i\}_{i=1}^M$:
                \[\begin{bmatrix}
        \phi_1(1)\overline{\phi_1} & \phi_2(1)\overline{\phi_2} & \phi_3(1)\overline{\phi_3} & \cdots& \phi_M(1)\overline{\phi_M}\\
        \phi_1(2)\overline{\phi_1} & \phi_2(2)\overline{\phi_2} & \phi_3(2)\overline{\phi_3} & \cdots &\phi_M(2)\overline{\phi_M}\\
        \vdots & \vdots & \vdots & & \vdots\\
        \phi_1(N)\overline{\phi_1} & \phi_2(N)\overline{\phi_2} & \phi_3(N)\overline{\phi_3} & \cdots &\phi_M(N)\overline{\phi_M}\\
        \end{bmatrix}.\]
 Notice that since $0\notin \{\phi_i\}_{i=1}^M$ we have that each $\phi_i$ contains at least one nonzero entry, say $\phi_i(k)\neq 0$. Then since $\phi_i(k)\overline{\phi_i}$ is part of $C_i$ we have that $C_i\neq 0$ for all $i$. 
        
        Now suppose that there exists scalars $a_i$ (not all zero) such that
        \[\sum_{i=1}^Ma_iC_i=0.\] 
        Then there is at least one $l$ such that $a_lC_l\neq 0$. Then by hypothesis, there is a row $k$ such that $\sum_{i} a_i\phi_i(k)\overline{\phi_i}=0$ but $a_l\phi_l(k)\overline{\phi_l}\neq 0$. Then 
        \[\sum_{i\in I_k}a_i\phi_i(k)\overline{\phi_i}=0\] which contradicts that $\{\phi_i\}_{i\in I_k}$ is linearly independent.

\end{proof}

\begin{remark}
        The conditions of the above proposition are fairly constrictive but, in certain cases, this can be useful. It will be used to verify a later example quickly.
\end{remark}

\begin{corollary}
        Let $\{\phi_i\}_{i=1}^M$ be a frame for which
         every subset of size $k$ is linearly independent. If the rows of the analysis operator are $k$-sparse then the induced outer products are linearly independent.
\end{corollary}

\section{Computation of Riesz Bounds}
In the following section we will examine more closely the Riesz bounds of the induced outer products. Here, we give the ``optimal'' Riesz bounds for a set of unit norm vectors, and sufficient conditions to achieve them.

The following is immediate by Lemma~\ref{pc2}. 

\begin{proposition}\label{orthorth}
        Let $\{\phi_i\}_{i=1}^M$ be vectors in $\HH^N$. The sequence $\{\phi_i\phi_i^*\}_{i=1}^M$ is orthonormal if and only if $\{\phi_i\}_{i=1}^M$ is orthonormal.
\end{proposition}

Since a redundant frame can not produce a Riesz sequence with tight Riesz bounds, one might ask how close we can get. Before computing the optimal Riesz bounds of a set of rank one projections we need to introduce the frame potential.

\begin{definition}
Let $\{\phi_i\}_{i=1}^M$ be a frame in $\HH^N$. The frame potential is
\[\mathrm{FP}(\{\phi_i\}_{i=1}^M)=\sum_{i=1}^M\sum_{j=1}^M|\ip{\phi_i}{\phi_j}|^2.\]
\end{definition}

\begin{proposition}
        The frame potential of a unit norm tight frame with $M$ elements in $\HH^N$ is $M^2/N$, which is a minimum over all unit norm frames.
\end{proposition}
See \cite{frame_potential,frames_for_undergraduates} for a proof of the above result.

\begin{theorem}\label{entf}
        If $\{\phi_i\}_{i=1}^M$ is a unit norm frame for $\HH^N$, then the upper Riesz bound of $\{\phi_i\phi_i^*\}_{i=1}^M$ is at least $M/N$. Moreover, we have equality if and only if $\{\phi_i\}_{i=1}^M$ is a unit norm tight frame.
\end{theorem}

\begin{proof}
If $\{\phi_i\}_{i=1}^M$ is a unit norm frame whose outer products have Gram matrix $G.$ Then
\begin{align*}
        \frac{M}{N} &\leq \frac{1}{M}\mathrm{FP}(\{\phi_i\}_{i=1}^M)\\
        &= \frac{1}{M} \left\|\left(\sum_{i=1}^M |\ip{\phi_i}{\phi_j}|^2\right)_{j=1}^M\right\|_{\ell_1}\\
        &\le \frac{1}{M} \sqrt{M}\left\|\left(\sum_{i=1}^M |\ip{\phi_i}{\phi_j}|^2\right)_{j=1}^M\right\|_{\ell_2}\\
        &= \left\|\left(\frac{1}{\sqrt{M}}\sum_{i=1}^M |\ip{\phi_i}{\phi_j}|^2\right)_{j=1}^M\right\|_{\ell_2}\\
        &=  \left\|G\left(\frac{1}{\sqrt{M}},\ldots, \frac{1}{\sqrt{M}}\right)^T\right\|_{\ell_2}\\
        &\le \|G\|\\
        & = \lambda_1
\end{align*}
where $\lambda_1$ is the largest eigenvalue of $G$. 

For the moreover part, if $\lambda_1=\frac{M}{N}$ then we have that $\frac{M^2}{N}=FP(\{\phi_i\}_{i=1}^M)$ so that $\{\phi_i\}_{i=1}^M$ is a unit norm tight frame. If on the other hand we have that $\{\phi_i\}_{i=1}^M$ is a unit norm tight frame, then 
\begin{align*}
\frac{M}{N} &= \frac{1}{M}FP(\{\phi_i\}_{i=1}^M)\\
        &= \frac{1}{M} \left\|\left(\sum_{i=1}^M |\ip{\phi_i}{\phi_j}|^2\right)_{j=1}^M\right\|_{\ell_1}\\
        &= \frac{1}{M} \left\|\left(\frac{M}{N},\dots,\frac{M}{N}\right)\right\|_{\ell_1}\\
        &= \frac{1}{M} \sqrt{M}\left\|\left(\frac{M}{N},\dots,\frac{M}{N}\right)\right\|_{\ell_2}\\
        &= \left\| \frac{1}{\sqrt{M}}\left(\frac{M}{N},\dots,\frac{M}{N} \right) \right\|_{\ell_2}\\
        &=  \left\|G\left(\frac{1}{\sqrt{M}},\dots, \frac{1}{\sqrt{M}}\right)\right\|_{\ell_2}\\
        &= \|G\|\\
        & = \lambda_1.
\end{align*}
\end{proof}

Now we will compute the optimal lower Riesz bounds
for outer product frames.

\begin{theorem}
        If $\{\phi_i\}_{i=1}^M$ is a unit norm frame for $\HH^N$, then the lower Riesz bound of $\{\phi_i\phi_i^*\}_{i=1}^M$ is at most $\displaystyle\frac{M(N-1)}{N(M-1)}$.
\end{theorem}
\begin{proof}
        Let $G$ be the Gram matrix of $\{\phi_i\phi_i^*\}_{i=1}^M$ with eigenvalues $\lambda_1\geq \lambda_2 \geq \cdots\geq \lambda_M$. Then
        $\mathrm{Tr}(G)=M$ gives 
        \[\sum_{i=2}^M\lambda_i=M-\lambda_1,\]
 Also,
 \[ (M-1)\lambda_M \le \sum_{i=2}^M \lambda_i,\]
 and so
        \[\lambda_{M}\leq \frac{\sum_{i=2}^M\lambda_i}{M-1}.\]
        Finally, we have
        
        \[\lambda_M \leq\frac{M-\lambda_1}{M-1}
        \leq \frac{M-\frac{M}{N}}{M-1}
        =\frac{M(N-1)}{N(M-1)}.\]
\end{proof}
        In the next theorem, we see that the above bounds are sharp. 
\begin{theorem}\label{Equiangular}
        Let $\{\phi_i\}_{i=1}^M$ be a unit norm equiangular
frame for $\HH^N$ with $M>N$ and let $c:=|\ip{\phi_i}{\phi_j}|^2$ for $i\neq j$. Then $\{\phi_i\phi_i^*\}_{i=1}^M$ is a Riesz sequence whose
Gram matrix has two distinct eigenvalues, both of which are
non-zero:
\[\lambda_1=1+(M-1)c\mbox{ and } \lambda_i = 1-c
\mbox{ for all }i=2,3,\ldots,M.\] 
Moreover, if $\{\phi_i\}_{i=1}^M$ is also a tight frame, then $c=\frac{M-N}{N(M-1)}$ and 
$\{\phi_i\phi_i^*\}_{i=1}^M$ is a Riesz sequence        
with Riesz bounds $\frac{M(N-1)}{N(M-1)}$, $\frac{M}{N}$.
\end{theorem}
Before proving the above result, we need a well known theorem (see e.g. \cite{matrix_analysis}).
\begin{theorem}[Sylvester's Determinant Theorem]
Let $S$ and $T$ be matrices of size $M\times N$ and $N\times M$ respectively. Then
\[\det (I_M+ST)=\det(I_N+TS).\]
\end{theorem}
\begin{proof}[Proof of Theorem~\ref{Equiangular}]
        Let $G$ be the Gram matrix for $\{\phi_i\phi_i^*\}_{i=1}^M$. Then
        \[G[i,j]=\left\{\begin{array}{ll}
        1 & \text{ if } i=j\\
        c & \text{ otherwise}
        \end{array}\right.\]
        Then we can write $G=(1-c)I_M+c1_{M}1_{M}^*$ and expand using Sylvester's determinant theorem with $S=1_M$ and $T=1_M^*$:
        \begin{align*}
                \det\left((1-c)I_M+c1_{M}1_{M}^* -\lambda I\right) & = \det\left((1-c-\lambda)I_M+c1_{M}1_{M}^*\right)\\
                &=(1-c-\lambda)^M\det\left(I_M+\frac{c}{1-c-\lambda}1_{M}1_{M}^*\right)\\
&=(1-c-\lambda)^M\det\left(I_1+\frac{c}{1-c-\lambda}1_{M}^*1_{M}\right)\\
                &=(1-c-\lambda)^{M-1}(1-c-\lambda+cM).
        \end{align*}
Setting the above equal to zero and solving for $\lambda$ we get the solutions $\lambda = 1-c$ occurring $(M-1)$-times and $\lambda=1+(M-1)c$ occurring once. 

If $c=0$, then $\{\phi_i\phi_i^*\}_{i=1}^M$ are orthonormal and hence so are $\{\phi_i\}_{i=1}^M$ contradicting the assumption that $M>N$. If $c=1$ then $\phi_i=\alpha_{ij}\phi_j$ with $|\alpha_{ij}|=1$ for all $i$ and $j$ contradicting the fact that this is a frame. Hence, $0<c<1$ and the outer products are Riesz. 

For the ``moreover'' part, we compute:
        \[1-c=1-\frac{M-N}{N(M-1)}=\frac{NM-N-M+N}{N(M-1)}=\frac{M(N-1)}{N(M-1)}\]
        and
        \[1+(M-1)c=1+(M-1)\frac{M-N}{N(M-1)}=\frac{N+M-N}{N}=\frac{M}{N}.\]
        
\end{proof}
        We can think of equiangular tight frames as minimizers of the the quantity $B-A$ where $A$ and $B$ are the Riesz bounds of the induced outer products. One problem is that there are few equiangular tight frames. If we want to produce an outer product sequence with arbitrary size and dimension and have predictably good bounds, we cannot use equiangular tight frames. At this time we do not know if there are other frames which achieve the optimal bounds above.

\section{Concrete Constructions of Riesz Bases of Outer Products}
 
Up to now, we have provided no concrete constructions of Riesz outer product sequences. We rectify this with the following examples.
\begin{example}\label{example_eiej}
Let $\{e_i\}_{i=1}^N$ be an orthonormal basis for $\RR^N$ and define $\{E_{ij}\}$ as follows
\[E_{ij}=\left\{\begin{array}{ll}
e_i & \text{ if } i=j \\
\frac{1}{\sqrt{2}}(e_i+e_j)& \text{ if } j>i
\end{array}\right.\]
for $i=1,\cdots,N$ and $i\leq j$. Then $\{E_{ij}E_{ij}^*\}$ is a Riesz basis for the space of symmetric operators in $\mathrm{sym}(\RR^{N\times N})$.
\end{example}
\begin{proof}
This follows immediately from Proposition~\ref{sparsity}.
\end{proof}
The following example provides an extension of the above to the complex case. It also provides a second (more intuitive) method of verifying that the above example is independent.

        \begin{example}\label{complex_eiej}
        Take $E_{ij}$ as before, and add the following 
        \[E_{ij}'=\frac{1}{2}(e_i+\sqrt{-1}e_j)(e_i+\sqrt{-1}e_j)^*\]
        for $j>i$. Then the resulting sequence is Riesz.
        \end{example}
        \begin{proof}
                Note that $E_{ij}'$ is a matrix with $1$ in the $(i,i)$ and $(j,j)$ entry and $-\sqrt{-1}$ in the $(i,j)$ entry and $\sqrt{-1}$ in the $(j,i)$ entry. 
                Then we know that $\sum_{i,j} a_{ij}E_{ij}+\sum_{i,j}a_{ij}'E_{ij}'=0$ if and only if the real and complex parts are $0$. We will do the real part and the complex part will follow immediately.
                $E_{ij}$ with $i\neq j$ is the square matrix with $1$'s in the $(i,i),(i,j),(j,i),$ and $(j,j)$ entry. Specifically, it is the only element in the sum for which the entries $(i,j)$ and $(j,i)$ could possibly be non-zero. Hence $a_{ij}=0$ for all $i\neq j$. The remaining terms $E_{ii}$ are orthonormal and hence $a_{ii}=0$ for all $i$. Thus the real part is independent and the complex part follows by the same argument.
        \end{proof}
        
        We know that the optimal Riesz bounds for a Riesz basis of outer products are $(N+1)/(N+2)$ and $(N+1)/2$. Using unit norm tight frames we can always achieve the upper bound. The lower bound is then the problem. Here we give a class of unit norm tight frames which produce nice lower bounds as well.

        \begin{example}\label{example_biangular}
        Let $\{\phi_i\}_{i=1}^{N+1}$ be the usual simplex 
        equiangular tight frame
         for $\RR^N$. Then consider the outer products
        \[\Phi_{ij}=\left(\frac{\phi_i+\phi_j}{\|\phi_i+\phi_j\|}\right)\left(\frac{\phi_i+\phi_j}{\|\phi_i+\phi_j\|}\right)^*\]
        for $j>i$. Then
        $\Phi_{ij}$ is Riesz provided $N\neq 3$ and has Riesz bounds $\frac{1}{2}$ and $\frac{N+1}{2}$ for $N\geq 7$.
        \end{example}
        \begin{proof}
        Barg et al. showed in \cite{kasso} that the frame 
        \[\frac{\phi_i+\phi_j}{\|\phi_i+\phi_j\|}\]
        is a unit norm tight
frame.  Hence by Theorem~\ref{entf} the upper Riesz bound of the induced outer products is
        \[\frac{N(N+1)}{2}\frac{1}{N}=\frac{N+1}{2}.\]

For the lower bound, we can consider the simplex in $\RR^N$ as  $\left\{\frac{Pe_i}{\|Pe_i\|}\right\}_{i=1}^{N+1}$ for where $\{e_i\}_{i=1}^{N+1}$ is an orthonormal basis for $\RR^{N+1}$, $P=I_{N+1}-ff^*$, and $f=\frac{1}{\sqrt{N+1}}\sum_{i=1}^{N+1}e_i$. Then we have
\begin{align*}
        \phi_i&=\frac{Pe_i}{\|Pe_i\|}\\
        &=\sqrt{\frac{N+1}{N}}\left(-\frac{1}{N+1},\dots,-\frac{1}{N+1},1-\frac{1}{N+1},-\frac{1}{N+1},\dots ,-\frac{1}{N+1}\right)
        \end{align*}
        
        and
        \begin{align*}
        \ip{\phi_i}{\phi_j}&=\frac{N+1}{N}\left(\frac{N-1}{(N+1)^2}-\frac{2}{N+1}\left(1-\frac{1}{N+1}\right)\right)\\
                &=-\frac{1}{N}.
        \end{align*}
        Now, $\|\phi_i+\phi_j\|^2=2\frac{N-1}{N}$ for $i\neq j$ and so we can compute the the Gram matrix of $\{\Phi_{ij}\}_{ij}$,
        \[G_\Phi[ij,kl]=\ip{\Phi_{ij}}{\Phi_{kl}}=\left\{\begin{array}{ll}
        1 & \text{ if } i=j \text{ and } k=l\\
        \frac{(N-3)^2}{4(N-1)^2} & \text{ if } i=k \text{ or } i=l \text{ or } j =k \text{ or } j=l \\
        \frac{4}{(N-1)^2} & \text{ if no indices are equal} 
        \end{array}\right..\]
        Consider the collection of unit norm vectors
        \[E_{ij}=\frac{1}{2}(e_i+e_j)(e_i+e_j)^* \text{ for } j>i\]
        and $\{e_i\}_{i=1}^{N+1}$ is an orthonormal basis for $\RR^{N+1}$. 
        Now its Gram matrix is
        \[G_E[ij,kl]=\left\{\begin{array}{ll}
        1 & \text{ if } i=j \text{ and } k=l\\
        \frac{1}{4} & \text{ if either } i=k \text{ or } i=l \text{ or } j =k \text{ or } j=l \\
        0 & \text{ if no indices are equal} 
        \end{array}\right..\]

        This gives us the decomposition 
        \begin{align*}
        G_\Phi&=\left(1-4\left(\frac{(N-3)^2}{4(N-1)^2}\right)\right)I_{N(N+1)/2}\\
        &\quad+4\left(\frac{(N-3)^2}{4(N-1)^2}-\frac{4}{(N-1)^2}\right)G_E\\
        &\quad+\frac{4}{(N-1)^2} 1_{N(N+1)/2}1_{N(N+1)/2}^*.
        \end{align*}
        Some inequalities,
        \[\frac{(N-3)^2}{4(N-1)^2}-\frac{4}{(N-1)^2}\geq 0\]
        if $N\geq 7$
        and
        \[1-4\left(\frac{(N-3)^2}{4(N-1)^2}\right)> 0\]
        if $N>2$. 
        The matrices $\left(1-\frac{(N-3)^2}{(N-1)^2}\right)I_{N(N+1)/2}$ and  $4\left(\frac{(N-3)^2}{4(N-1)^2}-\frac{4}{(N-1)^2}\right)G_E$ are positive-definite and $\frac{4}{(N-1)^2} 1_{N(N+1)/2}1_{N(N+1)/2}^*.$ is positive-semidefinite so
        \begin{align}
        \nonumber\lambda_{min}[G_\Phi]
        &\geq\lambda_{min}\left[\left(1-\frac{(N-3)^2}{(N-1)^2}\right)I_{N(N+1)/2}\right]\\
        \nonumber&\quad+\lambda_{min}\left[4\left(\frac{(N-3)^2}{4(N-1)^2}-\frac{4}{(N-1)^2}\right)G_E\right]\\
        \nonumber&\quad+\lambda_{min}\left[\frac{4}{(N-1)^2} 1_{N(N+1)/2}1_{N(N+1)/2}^*\right]\\
        \nonumber&=\left(1-\frac{(N-3)^2}{(N-1)^2}\right)\\
         &\quad+4\left(\frac{(N-3)^2}{4(N-1)^2}-\frac{4}{(N-1)^2}\right)\lambda_{min}\left[G_E\right]+0 \label{equation_blahblahblah}
        \end{align}
        We need to know $\lambda_{min}(G_E)$. 
        
        As in Example~\ref{example_eiej}, we will break up the sum. Let $E_{ij}=\frac{1}{2}(e_i+e_j)(e_i+e_j)^*$. Then 
        \begin{align*}
        \left\|\sum_{i=1}^{N+1}\sum_{j>i}a_{ij}E_{ij}\right\|^2&=\frac{1}{4}\sum_{i=1}^{N+1}\left[\abs{\sum_{j>i}a_{ij}+\sum_{j<i}a_{ji}}^2+2\sum_{j>i}\abs{a_{ij}}^2\right]\\
        &\geq \frac{1}{2}\sum_{j>i}\abs{a_{ij}}^2\\
        &=\frac{1}{2}
        \end{align*}
        for $a_{ij}$ which square sum to $1$.
        
        Then \eqref{equation_blahblahblah} becomes
        \[1-\frac{(N-3)^2}{4(N-1)^2}+2\left(\frac{(N-3)^2}{4(N-1)^2}-\frac{4}{(N-1)^2}\right)=\frac{N^2+2N-23}{2(N-1)^2}\geq \frac{1}{2}\]
        for $N\geq 6$. 
        
        Since these inequalities only hold for $N\geq 7$, we have computed the lower Riesz bounds for $N=2,3,\dots,6$ manually:
        \begin{center}
        \begin{tabular}{|l|l|}
        
        \hline $N$ & lower bound\\\hline\hline
                        $2$ &   $3/4$\\  \hline
                        $3$ &   $0$ \\\hline
                        $4$ &   $5/36$\\\hline
                        $5$ &   $3/8$\\\hline
                        $6$ &   $63/100$\\\hline
        \end{tabular}
        \end{center}
\end{proof}
\begin{remark}
        When $N=3$ we get another example of the strangeness of this problem. In this example we get that $\Phi_{14}=\Phi_{23}$ thus producing a dependent sequence.
\end{remark}

\section{Duals of Outer Products}

\begin{lemma}
Given a vector $\phi$ in $\HH^N$ and operators $T_1,T_2$ acting on
$\HH^N$ with $T_2$ symmetric, we have

(1) $T_1(\phi \phi^*)= (T_1\phi)\phi^*$.

(2)  $T_1(\phi\phi^*)T_2 = (T_1\phi)(T_2\phi)^*$.
\end{lemma}

\begin{proof}
(1)  We compute for $x \in \HH^N$
\begin{align*}
T_1(\phi\phi^*)(x) &= T_1(\langle x,\phi\rangle \phi)\\
&= \langle x,\phi\rangle T(\phi)\\
&= (T_1\phi)\phi^*(x).
\end{align*}

(2) We compute for $x\in \HH^N$
\begin{align*}
(\phi\phi^*)T_2(x) &= \langle T_2x,\phi\rangle \phi\\
&= \langle x,T_2\phi\rangle \phi\\
&= \phi (T_2\phi)^*(x).
\end{align*}
\end{proof}

\begin{proposition}\label{pc3}
If $\{\phi_i\}_{i=1}^M$ is a Riesz sequence in $\HH^N$ with biorthogonal vectors $\{\tilde{\phi}_i\}_{i=1}^M$, then the biorthogonal vectors for $\{\phi_i\phi_i^*\}_{i=1}^M$ are $\{P\tilde{\phi}_i\tilde{\phi}_i^*\}_{i=1}^M$ where $P$ is the orthogonal projection onto the span of $\{\phi_i\phi_i^*\}_{i=1}^M$.
\end{proposition}

\begin{proof}
We compute:
\[ \ip{\phi_i\phi_i^*}{P\tilde{\phi}_j\tilde{\phi}_j^*}_F = \ip{P\phi_i\phi_i^*}{\tilde{\phi}_j\tilde{\phi}_j^*}_F=
\left|\ip{\phi_i}{ \tilde{\phi}_j}\right|^2 = \delta_{ij}.\]
So the vectors  $\{P\tilde{\phi}_i\tilde{\phi}_i^*\}_{i=1}^M$ are biorthogonal to
$\{\phi_i\phi_i^*\}_{i=1}^M$.  
\end{proof}

\begin{remark}
        Projecting is necessary in the above proposition. For example, take $\{\phi_1,\phi_2\}$ to be a non-orthogonal Riesz basis for $\RR^2$. Then $\tilde{\phi}_1\perp \phi_2$ so take any $\psi_1\perp \phi_2$ with norm $1$ and scale $\tilde{\phi}_1$ so that $\ip{\phi_1}{\tilde{\phi}_1}=1$ i.e. $\tilde{\phi}_1=\frac{1}{\ip{\psi_1}{\phi_1}}\psi_1$. Then the Gram matrix of the induced outer products of $\{\phi_1,\phi_2,\tilde{\phi}_1\}$ 
        is
        \[\begin{bmatrix}
        1 & |\ip{\phi_1}{\phi_2}|^2 & 1\\
        |\ip{\phi_1}{\phi_2}|^2 & 1 & 0\\
        1 & 0 & 1
        \end{bmatrix}\]
        which has determinant $-|\ip{\phi_1}{\phi_2}|^4$. Since we have chosen $\phi_1\not \perp \phi_2$ this matrix is invertible. Hence these outer products are Riesz.  But then $\tilde{\phi}_1\tilde{\phi}_1^*$ is not in the span of the other two. Hence the projections are necessary. 
\end{remark}

\section{Outer Cross-Products of Frames}

        We now turn to a generalization of what we have done so far. Instead of considering $\{\phi_i\phi_i^*\}_{i=1}^M$, we will be examining the set of all rank one matrices obtainable through outer products. Specifically, for collections of vectors $\{\phi_i\}_{i=1}^M$ and $\{\psi_i\}_{i=1}^L$ in $\HH^N$, we will consider the collection $\{\phi_i\psi_j^*\}_{i=1,j=1}^{M\ , \ L}$. One immediate difference is that these outer products are no longer symmetric even if the original sequences are equal. As such, the ambient space is no longer the space of self-adjoint matrices, instead it is the space of all matrices of size $N\times N$. Another interesting aspect of considering such outer products is that the Gram matrix takes the form of another famous product in matrix theory.
        
        \begin{definition}
                Let $S=[s_{ij}]_{i,j}$ and $T$ be matrices of arbitrary size. The \emph{Kronecker product} of $S$ and $T$ is the block matrix
                \[S\otimes T=[s_{ij}T]_{ij}.\]
        \end{definition}
        
        \begin{lemma}
                Let $G_\phi$ and $G_\psi$ to be the Gram matrices of $\{\phi_i\}_{i=1}^M$ and $\{\psi_j\}_{j=1}^L$ respectively. The Gram matrix of $\{\phi_i\psi_j^*\}_{i=1,j=1}^{M\ , \ L}$ is $G_\phi\otimes G_\psi^T$.
        \end{lemma}
        \begin{proof}
                First note that
                \[\ip{\phi_i\psi_j^*}{\phi_k\psi_l^*}=\ip{\phi_i}{\phi_k}\ip{\psi_l}{\psi_j}\]
                which means that if we arrange our outer products
                \[\{\phi_1\psi_1^*,\phi_1\psi_2^*,\dots,\phi_1\psi_M^*,\phi_2\psi_1^*,\dots,\phi_M\psi_M^*\}\]
                then the Gram matrix of this collection of vectors is
                \[[\ip{\phi_i}{\phi_k}\ip{\psi_l}{\psi_j}]_{ij,kl}=G_\phi\otimes G_\psi^T.\]
        \end{proof}
        Now we are able to take advantage of another well known result from matrix theory, see \cite{HJ_topics}.
        
        \begin{theorem}
                Let $S$ and $T$ be square matrices with eigenvalues $\{\lambda_i\}_{i=1}^M$ and $\{\nu_i\}_{i=1}^L$ respectively. The eigenvalues of $S\otimes T$ are $\{\lambda_i\nu_j\}_{i=1,j=1}^{M\ , \ L}.$
        \end{theorem}
        We are ready for the fundamental theorem of outer cross-products.
        \begin{theorem}
        If $\{\phi_i\}_{i=1}^M$ and $\{\psi_j\}_{j=1}^L$ are collections of vectors in
$\HH^N$ which are:
        \begin{enumerate}
\item frames with frame bounds $A,B$ and $C,D$ respectively, then\\ $\{\phi_i\psi_j^*\}_{i=1,j=1}^{M\ , \ L}$ is a frame for
$\HH^{N\times N}$ with frame bounds $AC,BD$.
        \item Riesz sequences with Riesz bounds $A,B$ and $C,D$ then $\{\phi_i\psi_j^*\}_{i=1,j=1}^{M\ , \ L}$ is Riesz sequence
        for $\HH^{N\times N}$ with Riesz bounds $AC,BD$.
\end{enumerate}
\end{theorem}

\begin{proof}
        Let $G_\phi$ and $G_\psi$ be the Gram matrices of $\{\phi_i\}_{i=1}^M$ and $\{\psi_j\}_{j=1}^L$ respectively. Further suppose that the $G_\phi$ has eigenvalues $\{\lambda_i\}_{i=1}^M$ and $G_\psi$ has eigenvalues $\{\nu_i\}_{i=1}^L$. Assume that $\lambda_1\geq \lambda_2\geq \cdots \geq \lambda_M$ and $\nu_1\geq \nu_2\geq \cdots \geq \nu_L$. 
        
        If $\{\phi_i\}_{i=1}^M$ is a frame with frame bounds $A$ and $B$ then $A=\lambda_N$ and $B=\lambda_1$. Likewise, if $\{\psi_j\}_{j=1}^L$ is a frame with frame bounds $C$ and $D$ then $C=\nu_N$ and $D=\nu_1$. Then $G_\phi\otimes G_\psi^T$ has $N^2$ strictly positive eigenvalues and so $\{\phi_i\psi_j^*\}_{i=1,j=1}^{M\ , \ L}$ is a frame for $\HH^{N\times N}$. The frame bounds are the largest and smallest non-zero eigenvalues of $G_\phi\otimes G_\psi^T$ which are $BD$ and $AC$ respectively.
        
        If, on the other hand, $\{\phi_i\}_{i=1}^M$ and $\{\psi_j\}_{j=1}^L$ are Riesz sequences, then $\lambda_M>0$ and $\nu_L>0$ and so $\lambda_i\nu_j>0$ for all $i,j$. Hence $\{\phi_i\psi_j\}_{i=1,j=1}^{M\ , \ L}$ is Riesz with Riesz bounds $BD$ and $AC$.
\end{proof}

For the case of symmetric matrices
(see Proposition \ref{pc3}), to find the dual
functionals of a Riesz sequence of outer products, we had
to project the desired functionals onto the span of the
outer products.  Now we show that this assumption is not
necessary in the general case of {\it outer cross-products}.

\begin{theorem}
If $\{\phi_i\}_{i=1}^N$ and $\{\psi_i\}_{i=1}^N$ are Riesz bases in $\HH^N$ with 
dual Riesz bases $\{\tilde{\phi}_i\}_{i=1}^N$ and
$\{\tilde{\psi}_i\}_{i=1}^N$ respectively, then $\{\phi_i\psi^*_j\}_{i,j=1}^N$ is a Riesz bases for $\HH^{N\times N}$ with dual
basis $\{\tilde{\phi}_i\tilde{\psi}^*_j\}_{i,j=1}^N$.
\end{theorem}

\begin{proof}
We compute:
\[ \langle \tilde{\phi}_i\tilde{\psi}^*_j,\phi_l\psi_k^* \rangle_F
= \langle \phi_l,\phi_i\rangle \langle \psi_k,\psi_j\rangle
= \begin{cases} 1 & \text{if } l=i \text{ and } j=k\\
0 & \text{otherwise}
\end{cases}.\]
\end{proof} 

\section{Topological Properties of Independent Outer Product Sequences}
In the abstract, we make the claim that ``almost every" unit norm frame with a cardinality within a particular bound induces a set of independent outer products.

        In this section, we will consider the family of unit norm frames with cardinality $M\leq\dim\mathrm{sym}(\HH^{N\times N})$.
We see that we can identify this family with the topological space $ \bigotimes_{i=1}^M (S_{N-1})$.  We will use the standard metric for frames, $d(\Phi, \Psi) = \sqrt{ \sum_{i=1}^M \| \phi_i - \psi_i\|^2 }$, which is compatible with the subspace topology of the Euclidean topology with regards to $\bigotimes_{i=1}^M (S_{N-1})$.
Results of this kind are often done in frame theory using
algebraic geometry which might give a slightly stronger result
that the unit norm $M$-element frames which produce independent outer products
form an open dense set in the Zariski topology in the
family of all unit norm $M$-element frames.  We have
chosen not
to do this because only a fraction of the field knows
enough algebraic geometry to appreciate such results. 
Instead, we will give a direct, analytic construction for
the density of of the frames giving independent outer products.

\begin{lemma}\label{lem1}
If $\{\phi_i\}_{i=1}^N$ is a Riesz sequence in ${\mathbb H}^N$ with Riesz bounds
$A,B$ and 
\[ \sum_{i=1}^N \|\phi_i-\psi_i\|^2 < \varepsilon^2<A,\]
then $\{\psi_i\}_{i=1}^N$ is Riesz with Riesz bounds $(\sqrt{A}-\varepsilon)^2,
(\sqrt{B}+\varepsilon)^2$.
\end{lemma}

\begin{proof}
 For any $\{a_i\}_{i=1}^N$ we compute:
\begin{eqnarray*}
\|\sum_{i=1}^N a_i\psi_i\|&\le& \|\sum_{i=1}^Na_i\phi_i\|+\|\sum_{i=1}^Na_i(
\psi_i-\phi_i)\|\\
&\le& B^{1/2}\left ( \sum_{i=1}^N |a_i|^2\right )^{1/2}+\sum_{i=1}^N|a_i|\|\psi_i
- \phi_i\|\\
&\le& B^{1/2}\left ( \sum_{i=1}^N|a_i|^2 \right )^{1/2} + 
\left ( \sum_{i=1}^N|a_i|^2 \right )^{1/2}\left ( \sum_{i=1}^N\|\psi_i-\phi_i\|^2
\right )^{1/2}\\
&\le& (B^{1/2}+\varepsilon)\left ( \sum_{i=1}^N|a_i|^2 \right )^{1/2}.
\end{eqnarray*}
The stated upper Riesz bound is immediate from here.
The lower Riesz bound follows similarly.

\end{proof}

\begin{lemma}\label{lem2}
If $\|\phi\|=\|\psi\|=1$, then

\[ \|\phi\phi^*-\psi\psi^*\|_F^2 \leq 2\|\phi-\psi\|^2
.\]
\end{lemma}

\begin{proof}
We compute
\begin{align*}
\|\phi\phi^*-\psi\psi^*\|_F^2 &= \|\phi\phi^*\|_F^2 +\|\psi\psi^*\|_F^2 - 2 \langle \phi\phi^*,\psi\psi^*\rangle_F\\
&= 1+1-2|\langle \phi,\psi\rangle|^2\\
&= 2(1-|\langle \phi,\psi\rangle|^2)\\
&= 2(1-|\langle \phi,\psi \rangle|)(1+|\langle \phi,\psi\rangle|)\\
&= (2-2|\langle \phi,\psi \rangle|)(1+|\langle \phi,\psi\rangle|)\\
&\leq (2-2\mathrm{Re}\langle \phi,\psi \rangle)(1+|\langle \phi,\psi\rangle|)\\
&= (\|\phi\|^2 + \|\psi\|^2 -2\mathrm{Re}\langle \phi,\psi \rangle)(1+|\langle \phi,\psi \rangle|)\\
&= \|\phi-\psi\|^2(1+|\langle \phi,\psi \rangle|) \\
&\leq 2\|\phi-\psi\|^2.
\end{align*}
\end{proof}

\begin{proposition}\label{open}
Let $\{\phi_i\}_{i=1}^M$ are unit norm vectors in ${\mathbb H}^N$ with $\{\phi_i\phi_i^*\}_{i=1}^M$ a Riesz sequence 
having Riesz bounds $A,B$.  Given $0<\varepsilon<A/2$, choose
a unit norm set of vectors $\{\psi_i\}_{i=1}^M$ so that 
\[ \sum_{i=1}^M \|\phi_i-\psi_i\|^2 < \varepsilon<\frac{A}{2}.\]
Then $\{\psi_i\psi_i^*\}_{i=1}^M$ is
Riesz with Riesz bounds \[\left (\sqrt{A}-\sqrt{2\varepsilon}\right )^2\mbox{ and } \left (\sqrt{B}+\sqrt{2\varepsilon}\right )^2.\]
\end{proposition}

\begin{proof}
Assume the hypotheses.
It follows from our Lemma~\ref{lem2} that
\[ \sum_{i=1}^M \|\phi_i\phi_i^*-\psi_i\psi_i^*\|_F^2 \le 2\sum_{i=1}^M\|\phi_i-\psi_i\|^2<2\varepsilon\]
Now by Lemma~\ref{lem1} we have that $\{\psi_i\psi_i^*\}_{i=1}^M$ is Riesz with Riesz bounds
\[ \left (\sqrt{A}-\sqrt{2\varepsilon}\right )^2,\ \left (\sqrt{B}+\sqrt{2\varepsilon}\right )^2.\]

\end{proof}
The above proposition says that the set of frames with cardinality $M\leq \dim \mathrm{sym}(\HH^{N\times N})$ is open in $ \bigotimes_{i=1}^M (S_{N-1})$. In the remainder of this section we will show that this set is also dense. While other authors have studied the density of outer products in terms of commutative algebra \cite{Jameson_thesis}, here we show this fact constructively and quantitatively using only standard analytic and Euclidean topological notions.

\begin{lemma}\label{invertible_riesz}
                Let $S$ be an invertible operator and suppose $\{\phi_i\}_{i=1}^M$ are vectors in $\HH^N$. Then $\{\phi_i\phi_i^*\}_{i=1}^M$ is independent if and only if $\{S\phi_i(S\phi_i)^*\}_{i=1}^M$ is independent. 
        \end{lemma}
        \begin{proof}
                Let $\{a_i\}_{i=1}^M$ be scalars, not all zero. We have
                \[0=\sum_{i=1}^Ma_i\phi_i\phi_i^*\]
                if and only if 
                \[0=S\left(\sum_{i=1}^Ma_i\phi_i\phi_i^*\right)S^*=\sum_{i=1}^Ma_i(S\phi_i)(S\phi_i)^*.\]
        \end{proof}
        
        Now we construct a large family of bases of outer products. 
        
        \begin{lemma}\label{small_basis_construction}
                Given a unit norm vector $\psi\in \HH^N$,
               $\varepsilon >0$, there is a
                unit norm basis for $\mathrm{sym}(\HH^{N\times N})$ consisting of outer products $\{\phi_i\phi_i^*\}_{i=1}^d$ with $d={\dim\mathrm{sym}(\HH^{N\times N})}$,  such that $\|\phi_i-\psi\|^2<\varepsilon$ for all $i=1,\dots , d$. 
        \end{lemma}
        \begin{proof}
First, we will assume that we have a
unit norm basis $\{\psi_i\psi_i^*\}_{i=1}^d$
of $\mathrm{sym}(\HH^{N\times N})$
with $ \langle \psi,\psi_i\rangle
>0$ for all $i$ and $\psi=e_1$ for an orthonormal basis
$\{e_j\}_{j=1}^N$ of $\HH^N$. We can see that such a basis exist by a unitary transformation of Example~\ref{example_eiej} or Example~\ref{complex_eiej}. Choose $\delta >0$
with the following property:  If 
\[ S=diag(1,\delta,\delta,\ldots,\delta),\]
then for all $i=1,2,\ldots,d$ we have
\begin{eqnarray}\label{Eqn1}
 \sum_{j=2}^N |S\psi_i(j)|^2 = \delta^2 \sum_{j=2}^N|\psi_i(j)|^2 \le \frac{\varepsilon}{2} |\psi_i(1)|^2 \le 
\frac{\varepsilon}{2} \|S\psi_i\|^2.
\end{eqnarray}
Let 
\[ \phi_i = \frac{S\psi_i}{\|S\psi_i\|}\mbox{ for all }
i=1,2,\ldots,d,\] 
and observe that $\|\phi_i\|=1$ and Equation \ref{Eqn1} 
imply
\[ \phi_i(1) \ge 1-\frac{\varepsilon}{2}.\]
Now we compute for all $i=1,2,\ldots,d$
\[ \|\psi-\phi_i\|^2  = |1- \phi_i(1)|^2 + 
\sum_{j=2}^N |\phi_i(j)|^2\le  \epsilon.
\]
Since $\{\psi_i\psi_i^*\}_{i=1}^d$ is linearly independent,
by Lemma \ref{invertible_riesz}, the $\{\phi_i\phi_i^*\}_{i=1}^d$ are also independent.

For the general case, given $\psi$ and $\{\psi_i\}_{i=1}^d$
with independent outer products,
choose a vector $\phi$ so that
$\langle \phi,\psi_i\rangle \not= 0$ for all $i=1,2,\ldots,
d$.  By replacing $\phi$ by $c_i\phi$ with $|c_i|=1$ if necessary, we can assume these
inner products are all strictly positive.  By the above,
we can find $\{\phi_i\}_{i=1}^d$ with their outer products
independent and
\[ \|\phi-\phi_i\|^2<\varepsilon.\]
Choose a unitary operator $U$ so that $U\phi=\psi$ and we
have
\[ \|\psi - U\phi_i\|^2 = \|U\phi -U\phi_i\|^2
= \|\phi-\phi_i\|^2 <
\varepsilon.\]
This completes the proof.  
        \end{proof}

        With the above lemmas we are ready to prove the following.

        \begin{theorem}
                The set of all frames $\{\phi_i\}_{i=1}^M$ with $M\leq \dim\mathrm{sym}(\HH^N)$ which produce independent outer products is open and dense in the family of $M$-element frames. 
        \end{theorem}
        \begin{proof}
        
        This set was already shown to be open by Proposition~\ref{open}. All that remains to show is that this set is also dense. Let $\phi_1'=\phi_1$ and proceed by induction. Assume that we have a collection of vectors $\{\phi_i'\}_{i=1}^{M_0}$ such that $\|\phi_i'-\phi_i\|<\varepsilon/M$ for all $i=1,\dots,M_0$ and $\{\phi_i'(\phi_i')^*\}_{i=1}^{M_0}$ is independent. Then by Lemma~\ref{small_basis_construction} there exists a unit norm basis $\{\psi_i\}_{i=1}^{\dim\mathrm{sym}(\HH^{N\times N})}$ such that $\|\phi_{M_0+1}-\psi_i\|<\varepsilon/M$ for all $i$.  Since dim span $\{\psi_i \psi_i^*\}_{i=1}^{M_0} = M_0$, we can  
        choose $\phi_{M_0+1}'=\psi_k$ such that $\psi_k\psi_k^*\notin \mathrm{span}(\{\phi_i'(\phi_i')^*\}_{i=1}^{M_0})$. Then the set $\{\phi_i'\}_{i=1}^{M_0+1}$ induces independent outer products with $\|\phi_i'-\phi_i\|<\varepsilon/M$ for all $i$. By induction, we have obtained a set $\{\phi_i'\}_{i=1}^{M}$ such that
        \[\sum_{i=1}^M\|\phi_i'-\phi_i\|<\varepsilon\]
    and which induces independent outer products.       
%

        \end{proof}


\section{A Geometric Classification of All Finite Dependent Outer Product Sequences}

We will now precisely classify all frames that induce dependent outer products in terms of compact manifolds within finite dimensional Hilbert spaces.  This itself is reliant upon some results regarding positive semi-definite matrices, which are given in the final section.\\

It should be added that we are interested in classifying \textit{dependent} sets;  as we have seen in the previous section, these are far less common than independent sets.

\subsection{Some Necessary and Sufficient Conditions}
%
This section heavily relies on the following theorem, which will be proven in Section~\ref{big_ol_proof}.
\begin{theorem}
\label{samerankfamily1}
Let $T$ be a $N \times N$ positive semi-definite matrix.  Let $\{e_i\}_{i=1}^N$ be the eigenvectors of $T$ with the corresponding eigenvalues $\{\lambda_i\}_{i=1}^N$.  Let $I_+ \subset \{1, \ldots, N\}$ be the index for the eigenvectors with positive eigenvalues, i.e., $i \in I_+ \Leftrightarrow \lambda_i > 0$.\\
Let $\{ a_i \}_{i \in I_+}$ be a sequence of scalars such that $\sum_{i \in I_+} |a_i|^2 = 1$.  Then, for the vector $v = \sum_{i \in I_+} a_i \sqrt{\lambda_i} e_i$ , we will have:

$$rank\ \begin{bmatrix} T & v \\ v^* & 1 \end{bmatrix} = rank\ T$$

Likewise, the converse is true:  if we have $rank\ \begin{bmatrix} T & v \\ v^* & 1 \end{bmatrix} = rank\ T$, then  $v = \sum_{i \in I_+} a_i \sqrt{\lambda_i} e_i$ for some collection of scalars indexed by $I_+$, $\{a_i\}_{i \in I_+}$ where $\sum_{i \in I_+} |a_i|^2 = 1$.  
\end{theorem}

\begin{proposition}
\label{bsimplexprop}
Let $\{\phi_i\}_{i=1}^M\subset \HH^N$ be unit norm, and add an additional unit norm vector $\phi_{M+1}$.  Assume the set of induced outer products is $\{ \phi_i \phi_i^* \}_{i=1}^{M}$ is independent, and that $M + 1 \le \mathrm{dim}\ \mathrm{sym} (\HH^{N \times N})$.\\
Let $G_{op}$ be the Gram matrix of the induced outer products for the original sequence, that is the Gram matrix of $\{\phi_i \phi_i^* \}_{i=1}^M$, and denote the eigenvectors of $G_{op}$ as $\{ e'_i : 1 \le i \le M \}$ and the associated eigenvalues $\{\lambda_i' : 1 \le i \le M \}$.  \\
We consider the analysis operator $T$ for $\{\phi_i\}_{i=1}^M$ acting on $\phi_{M+1}$.  This is  $$T \phi_{M+1} = \begin{bmatrix} \langle \phi_{M+1}, \phi_1 \rangle \\ \langle \phi_{M+1}, \phi_2 \rangle \\ \vdots \\ \langle \phi_{M+1}, \phi_M \rangle \end{bmatrix}$$\\
Consider the following second order elliptic function:\\

\begin{equation}
\label{ellipticeqn1}
f(x_1, x_2, \ldots, x_M) = \sum_{1 \le i \le M} \frac{|x_i|^2}{\lambda_i'}
\end{equation}

Let $y_1 e'_1 + y_2 e'_2 + \cdots + y_M e'_M = T \phi_{M+1} \circ \overline{T \phi_{M+1}}$ be the representation of $T \phi_{M+1} \circ \overline{T \phi_{M+1}}$ within $\{e'_1, \ldots, e'_M  \}$.  Then we will have that $f(y_1, \ldots, y_M) = 1$ if and only if $\{\phi_i \phi_i^* \}_{i=1}^{M+1}$ is a dependent set.  
\end{proposition}
\begin{proof}
This follows directly from Theorem~\ref{samerankfamily1} and the identity of $G_{op} = G \circ \overline{G}$.  If we add the additional vector $\phi_{M+1}$ to our basis then the ${(M+1)}^{th}$ column of the Gram matrix for the outer products  $\{\phi_i \phi_i^*\}_{i=1}^{M+1}$ is \[\begin{bmatrix} T \phi_{M+1}\ \circ\ \overline{T \phi_{M+1}} \\ 1 \end{bmatrix},\] while the ${(M+1)}^{th}$ row is $[(T \phi_{M+1} \circ \overline{T \phi_{M+1}})^* \ \ 1]$.  We know that the dimension spanned by a frame is exactly the rank of its Gram matrix;  Theorem~\ref{samerankfamily1} implies that $T \phi_{M+1} \circ \overline{ T \phi_{M+1}\ }$ must precisely meet the criteria of this proposition to have the condition that the rank of the Gram matrix does not increase, and thereby does not increase the dimension spanned by the set $\{ \phi_i \phi_i^* \}_{i=1}^{M+1}$, i.e., this collection of outer products produces a dependent set. 
\end{proof}

\begin{remark}
\label{forthordervarietyrmk}
The previous theorem yields a quartic algebraic variety/manifold that will come in handy.  Let $\{e'_i\}_{i=1}^M$ be as in the theorem.  Consider the quartic equation for $v \in \HH^M$:
\begin{equation}
\label{forthordervariety}
\sum_{i=1}^M \frac{|\langle v \circ \overline{v} , e'_i \rangle |^2}{\lambda_i} = 1
\end{equation}

We use the notation $\mu^4_{\{\phi_i\}_{i=1}^M}$ to signify this quartic manifold embedded in $\HH^M$.  Note that  $v = T \phi_{M+1}$ satisfies this equation if and only if $\phi_{M+1}$ satisfies the criteria for the previous theorem.  Thus, if we are to consider the forth order algebraic variety for all $v \in \HH^M$\ that satisfy this equation, then the collection of all $T \phi_{M+1}$ such that $\phi_{M+1}$ satisfy the criteria for the  previous theorem are contained entirely within this variety.

\end{remark}

\section{Full Geometric Characterization of Dependent Outer Products}
Without loss of generality, we order every frame in this section such that $\{\phi_1, \ldots, \phi_N\}$ is a basis for its Hilbert space $\HH^N$, and $\{\phi_1, \ldots, \phi_{M_0} \}$ with $M_0 \le M$ such that $\{\phi_1 \phi_1^*, \ldots, \phi_{M_0} \phi_{M_0}^* \}$ is an independent sequence within the induced set of outer products $\{\phi_i \phi_i^* \}_{i=1}^M$.  Unless otherwise noted, we assume $M \le \mathrm{dim}\ \mathrm{sym}(\HH^{N \times N})$.  By default, $T$ will be the analysis operator for the frame $\{\phi_i \}_{i=1}^M$, while $S_{N-1}$ is be the unit sphere in $\HH^N$.\\

We start with some necessary lemmas.

\begin{lemma}
\label{elliplem1}
Let $S_{N-1}$ be the unit sphere in $\HH^N$.  $T S_{N-1}$ is an ellipsoid  embedded within $\HH^M$ with a Euclidean surface of dimension $N-1$;  moreover, $T$ is injective from $S_{N-1} \mapsto T S_{N-1}$.  
\end{lemma}
\begin{proof}
\label{elliplem2}
By lemma 3.24 of \cite{frames_for_undergraduates}, we know that $T$ is injective on $\HH^N$;  limiting its domain to $S_{N-1}$ retains injectivity.  If we limit the codomain to the range of $T$, so that we have the mapping $T: \HH^N \mapsto Range\ T$, then we have that $T S_{N-1}$ is an ellipsoid in an $N$-dimensional subspace of $\HH^M$ (see chapter 7 of \cite{frames_for_undergraduates}).  If we expand the codomain to $\HH^M$, we have an $N-1$ dimensional ellipsoidal manifold embedded in $\HH^M$. 
\end{proof}

\begin{remark}
\label{invT}
We use the notation ``$T^{-1}$" to indicate the inverse of the bijection $T \big{|}_{S^{N-1}}$, as above.
\end{remark}

\begin{lemma}\label{lemma100}
Let $G$ be the Gram matrix of our frame.  Arrange the eigenvalues of $G$ so that $\lambda_1 \ge \lambda_2 \ge \cdots \ge \lambda_N > 0$ and $\lambda_j = 0$ for $N < j \le M$, and denote the corresponding eigenvectors with $\{e_i\}_{i=1}^M$.  Then the ellipsoid $T S_{N-1}$ is the set of vectors $v = v_1 e_1 + \cdots v_N e_N$ where $\sum_{i=1}^N \frac{|v_i|^2}{\lambda_1} = 1$.
\end{lemma}

\begin{proof}
This again follows from the Lemma \ref{lemma100} and Theorem \ref{samerankfamily1}.
\end{proof}

\begin{remark}
\label{manifoldsremark}
For a given frame, we denote the quartic manifold given by the Gram matrix of outer products implicitly stated in theorem \ref{bsimplexprop} and explicitly stated in the following remark \ref{forthordervarietyrmk} as $\mu^4_{\{ \phi_i \}_{i=1}^M}$;  denote the second order (elliptic) manifold in lemmas \ref{elliplem1} and \ref{elliplem2} as $\mu^2_{\{\phi_i\}_{i=1}^M}$.
\end{remark}

\subsection{A Characterization of All Frames That Yield Dependent Outer Products with Cardinality less than $\mathrm{dim}\ \mathrm{sym} (\HH^{N \times N}) $ }

\begin{theorem}
\label{classificationtheorem}
Let $M < \mathrm{dim}\ \mathrm{sym}(\HH^{N \times N})$.  If $\{\phi_i \phi_i^* \}_{i=1}^M$ is independent, the set of vectors in $S_{N - 1}$ that will yield a dependent set of outer products will be $T^{-1}(\mu^2_{\{\phi_i\}_{i=1}^M} \cap\mu^4_{\{\phi_i\}_{i=1}^M} )$, which will be compact in the Euclidean topology.
\end{theorem}
\begin{proof}
Remembering the notation from remark~\ref{manifoldsremark}, we see that $\mu^2_{\{\phi_i\}_{i=1}^M} \cap\mu^4_{\{\phi_i\}_{i=1}^M}$ are exactly the portion of the image of $T$ that corresponds to the dependent outer products.  Since the manifolds $\mu^2_{\{\phi_i\}_{i=1}^M}$ and $\mu^4_{\{\phi_i\}_{i=1}^M}$ are closed and bounded within a Euclidean space, they are compact and likewise their intersection $\mu^2_{\{\phi_i\}_{i=1}^M} \cap\mu^4_{\{\phi_i\}_{i=1}^M}$ is compact.  By the injectivity of $T$ on $S_{N-1}$ and remark \ref{invT}, we see that $T^{-1} (\mu^2_{\{\phi_i\}_{i=1}^M} \cap\mu^4_{\{\phi_i\}_{i=1}^M})$ forms a compact subset of $S_{N-1}$.
\end{proof}

\subsection{A Geometric Result}
While it is beyond the scope of this paper to fully analyze this, we find that carrying this on for frames with induced outer product sets of dimensionality equal to $\mathrm{dim}\ \mathrm{sym} ( \HH^{N \times N} )$ yields a possibly interesting geometric result due to the loss of independence in the induced outer products.

\begin{proposition}
Suppose that $\{\phi_i\}_{i=1}^M$ is a unit norm frame for $\HH^N$ where $\mathrm{dim} \ \mathrm{span} \{ \phi_i \phi_i^* \}_{i=1}^M = \mathrm{dim}\ \mathrm{sym} (\HH^{N \times N})$.  Then  $\mu^2_{\{\phi_i\}_{i=1}^M} \subseteq\mu^4_{\{\phi_i\}_{i=1}^M}   $.
\end{proposition}

\begin{proof}
We already know that if we expand the frame $\{\phi_i\}_{i=1}^M$ to the point where \emph{any} additional vector $v \in S_{N-1}$ induces a \emph{dependent} outer product sequence $\{\phi_i \phi_i^*\}_{i=1}^M \cup \{v v^*\}$, we will have $T v \in \mu^4_{\{\phi_i\}_{i=1}^M}$.  But this implies that $T^{-1} (\mu^2_{\{\phi_i\}_{i=1}^M} \cap \mu^4_{\{\phi_i\}_{i=1}^M}  ) = T^{-1}(\mu^2_{\{\phi_i\}_{i=1}^M}) = S_{N-1}$.  The conclusion follows.
\end{proof}

\begin{remark}
This gives us an instance where an elliptic manifold with a surface that is locally Euclidean of dimension $(N - 1)$ embedded within $\HH^M$, which is contained entirely within a fourth order manifold of dimension $(M - 1)$ also embedded within the same $\HH^M$, where $M>N$.
\end{remark}

\section{Expanding Positive Semi-Definite Matrices While Preserving Rank}

\subsection{Main Theorem on Positive Semi-Definite Matrices}\label{big_ol_proof}
Now we prove Theorem~\ref{samerankfamily1}. We prove this Theorem in the form of two propositions (``forwards" and ``converse").  Likewise, we prove several lemmas for each proposition.

\subsection{Necessary Lemmas for ``Forwards" Proposition}

\begin{lemma}
Let $T$ be an $N \times N$ positive semi-definite matrix with eigenvector $e_i$ and associated eigenvalue $\lambda_i > 0$.  Then we will have

$$rank\ \begin{bmatrix} T & \sqrt{\lambda_i} e_i \\ (\sqrt{\lambda_i} e_i )^* & 1 \end{bmatrix} = rank\ T$$.
\end{lemma}

\begin{proof}
By the spectral theorem, we know that $T$ has $N$ eigenvectors $\{e_j\}_{j=1}^N$ with real-valued eigenvalues $\{\lambda_j\}_{j=1}^N$, and we have the representation $T = \sum_{j = 1}^N \lambda_j P_j$, where $P_j$ is the projection onto $e_j$.  Since by the hypothesis $\lambda_i > 0$, we have $T \left( (1/\sqrt{\lambda_i}) e_i \right) = \sqrt{\lambda_i} e_i$  This means that $\sqrt{\lambda_i} e_i \in Range\ T$.  Thus, $rank\ T = rank\ \begin{bmatrix} T \\ (\sqrt{\lambda_i} e_i)^* \end{bmatrix}$.\\

To complete the lemma, we show that the existence of a vector $w$ such that $\begin{bmatrix} T \\ (\sqrt{\lambda_i} e_i)^* \end{bmatrix} w = \begin{bmatrix} \sqrt{\lambda_i} e_i \\ 1 \end{bmatrix}$.  We set $w = \sqrt{\lambda_i} e_i$;  this yields

$$ \begin{bmatrix} T \\ (\sqrt{\lambda_i} e_i)^* \end{bmatrix} w = \begin{bmatrix} \sqrt{\lambda_i} e_i \\ 1 \end{bmatrix} $$

So we have that $\begin{bmatrix} \sqrt{\lambda_i} e_i \\ 1 \end{bmatrix} \in Range\ \begin{bmatrix} T \\ (\sqrt{\lambda_i} e_i)^* \end{bmatrix}$;  this implies that $$rank\ \begin{bmatrix} T & \sqrt{\lambda_i} e_i \\ (\sqrt{\lambda_i} e_i)^* & 1 \end{bmatrix} = rank\ \begin{bmatrix} T \\ (\sqrt{\lambda_i} e_i)^* \end{bmatrix}.$$

By our prior result we can conclude:

$$rank\ \begin{bmatrix} T & \sqrt{\lambda_i} e_i \\ (\sqrt{\lambda_i} e_i )^* & 1 \end{bmatrix} = rank\ T.$$
\end{proof}

\begin{lemma}
\label{poseigs1}
Let $T$ be an $N \times N$ positive semi-definite matrix, with distinct eigenvectors $e_i$ and $e_j$ with positive eigenvalues.
Then, for any two scalars $a,b$ such that $|a|^2 + |b|^2 = 1$, we will have:

$$rank\ T = rank\ \begin{bmatrix} T & a\sqrt{\lambda_i} e_i  + b \sqrt{\lambda_j} e_j \\ (a\sqrt{\lambda_i} e_i  + b \sqrt{\lambda_j} e_j)^* & 1 \end{bmatrix}$$
\end{lemma}
\begin{proof}
We proceed as in the prior theorem.  First, we check that $rank\ T = rank\ [T \ \ (a\sqrt{\lambda_i} e_i  + b \sqrt{\lambda_j} e_j)]$.  We see that $T\left( (a/\sqrt{\lambda_i}) e_i  + (b /\sqrt{\lambda_j}) e_j \right) = (a\sqrt{\lambda_i} e_i  + b \sqrt{\lambda_j} e_j) \in Range\ T$, which yields $$rank\ T = rank\ [T \ \ (a\sqrt{\lambda_i} e_i  + b \sqrt{\lambda_j} e_j)] = rank\ \begin{bmatrix} T \\ (a\sqrt{\lambda_i} e_i  + b \sqrt{\lambda_j} e_j)^* \end{bmatrix}$$

We can now see that $$\begin{bmatrix} T \\ (a\sqrt{\lambda_i} e_i  + b \sqrt{\lambda_j} e_j)^* \end{bmatrix} \left( (a/\sqrt{\lambda_i}) e_i  + (b /\sqrt{\lambda_j}) e_j \right) = \begin{bmatrix} (a\sqrt{\lambda_i} e_i  + b \sqrt{\lambda_j} e_j) \\ 1 \end{bmatrix}$$

which implies $$\begin{bmatrix} (a\sqrt{\lambda_i} e_i  + b \sqrt{\lambda_j} e_j) \\ 1 \end{bmatrix} \in Range\ \begin{bmatrix} T \\ (a\sqrt{\lambda_i} e_i  + b \sqrt{\lambda_j} e_j)^* \end{bmatrix}$$

and so we have 
\begin{align*}
& rank \begin{bmatrix} T \\ (a\sqrt{\lambda_i} e_i  + b \sqrt{\lambda_j} e_j)^*  \end{bmatrix}=\\
& 
rank \begin{bmatrix} T & a\sqrt{\lambda_i} e_i  + b \sqrt{\lambda_j} e_j \\ (a\sqrt{\lambda_i} e_i  + b \sqrt{\lambda_j} e_j)^* & 1 \end{bmatrix}
\end{align*}

the conclusion directly follows.
\end{proof}

\subsection{First proposition}
This is the ``forwards" implication of theorem (\ref{samerankfamily1})  (``$\Rightarrow$").
\begin{proposition}
\label{forwardsprop}
Let $T$ be a $N \times N$ positive semi-definite matrix.  Let $\{e_i\}_{i=1}^N$ be the eigenvectors of $T$ with the corresponding eigenvalues $\{\lambda_i\}_{i=1}^N$.  Let $I_+ \subset \{1, \ldots, N\}$ be the index for the eigenvalues with positive eigenvectors, i.e., $i \in I_+ \Leftrightarrow \lambda_i > 0$.\\
Let $\{ a_i \}_{i \in I_+}$ be a sequence of scalars such that $\sum_{i \in I_+} |a_i|^2 = 1$.  Then, for the vector $v = \sum_{i \in I_+} a_i \sqrt{\lambda_i} e_i$ , we will have:

$$rank\ \begin{bmatrix} T & v \\ v^* & 1 \end{bmatrix} = rank\ T$$
\end{proposition}

\begin{proof}
This is just an extension of lemma (\ref{poseigs1}) to an arbitrary number of eigenvectors.  Let $\{a_i\}_{i \in I_+}$ be a collection of scalars such that $\sum_{i \in I_+} |a_i|^2 = 1$.  We first see that $T \left( \sum_{i \in I_+} a_i \frac{1}{\sqrt{\lambda_i}} e_i \right) = \sum_{i \in I_+} a_i {\sqrt{\lambda_i}} e_i$.  This means that $rank\ T = rank\ [T \ \  \sum_{i \in I_+} a_i {\sqrt{\lambda_i}} e_i ] = rank\ \begin{bmatrix} T \\ \left( \sum_{i \in I_+} a_i {\sqrt{\lambda_i}} e_i \right)^* \end{bmatrix}$.\\
We see that $$\begin{bmatrix} T \\ \left( \sum_{i \in I_+} a_i {\sqrt{\lambda_i}} e_i \right)^* \end{bmatrix} \left( \sum_{i \in I_+} a_i \frac{1}{\sqrt{\lambda_i}} e_i \right)$$ 

$$= \begin{bmatrix} T \left( \sum_{i \in I_+} a_i \frac{1}{\sqrt{\lambda_i}} e_i \right) \\ \left\langle \left( \sum_{i \in I_+} a_i {\sqrt{\lambda_i}} e_i \right),  \left( \sum_{i \in I_+} a_i \frac{1}{\sqrt{\lambda_i}} e_i \right) \right\rangle \end{bmatrix} $$

$$= \begin{bmatrix} \left( \sum_{i \in I_+} a_i {\sqrt{\lambda_i}} e_i \right) \\ \sum_{i \in I_+} a_i \overline{a_i} \frac{\sqrt{\lambda_i}}{\sqrt{\lambda_i}} \end{bmatrix} = \begin{bmatrix} \left( \sum_{i \in I_+} a_i {\sqrt{\lambda_i}} e_i \right) \\ 1 \end{bmatrix}$$

This implies that the vector $\begin{bmatrix} \left( \sum_{i \in I_+} a_i {\sqrt{\lambda_i}} e_i \right) \\ 1 \end{bmatrix}$
is within the range of the matrix $\begin{bmatrix} T \\ \left( \sum_{i \in I_+} a_i {\sqrt{\lambda_i}} e_i \right)^* \end{bmatrix}$.  This will give us

$$rank\ \begin{bmatrix} T \\ \left( \sum_{i \in I_+} a_i {\sqrt{\lambda_i}} e_i \right)^* \end{bmatrix} = rank\ \begin{bmatrix} T & \left( \sum_{i \in I_+} a_i {\sqrt{\lambda_i}} e_i \right) \\ \left( \sum_{i \in I_+} a_i {\sqrt{\lambda_i}} e_i \right)^* & 1 \end{bmatrix}$$

The conclusion will follow.
\end{proof}

\subsection{Necessary Lemmas for Converse Proposition}

\begin{observation}
\label{signedspaces}
Let $T$ be a positive semi-definite matrix on $\HH^N$.  By the spectral theorem, $T = \sum_{i=1}^N \lambda_i P_i$, where $P_i$ is a projection onto the eigenvector $e_i$ with the associated real eigenvalue $\lambda_i$. \\ 

We can partition $\HH^N$ into two orthogonal subspaces, $V_{0}$ and $V_+$, where $V_+ = span\ \{e_i :\ \lambda_i > 0, \ 1 \le i \le N \  \}$, and $V_0 = span\ \{e_i :\ \lambda_i = 0, \  1 \le i \le N \}$.\\

Notice the orthogonality of the eigenvectors transfers to these spaces: $\HH^N = V_0 \oplus V_+$.)
\end{observation}

\begin{lemma}
\label{nulllemma}
Let $T$ be a positive-semi-definite matrix on $\HH^N$, and let $V_0$ be as in observation (\ref{signedspaces}).  Let $v \in \HH^N$.\\

If $P_{V_0} v \neq 0$, then $rank\ \begin{bmatrix} T & v \\ v^* & 1 \end{bmatrix} > rank\ T$.
\end{lemma}
\begin{proof}
Since $\HH^N = V_0 \oplus V_+$, we have that $V_0^\perp = V_+$.\\

We note that $ker\ T = V_0$, and $Range\ T = V_+$.  If $v \in \HH^N$, then $v = P_{V_+ } v + P_{V_0} v$;  if $P_{V_0} v \neq 0$, then $P_{V_0} v \notin Range\ T$ and hence $v \notin Range\ T$.  It follows that $rank\ [T \ \ v] > rank\ T$, and that $rank\ \begin{bmatrix} T & v \\ v^* & 1 \end{bmatrix} \ge rank\ [T \ \ v] > rank\ T$.
\end{proof}

\begin{lemma}
\label{negativelemma}
Let $T$ be a positive semi-definite matrix on $\HH^N$, and let $V_0$ be as in observation (\ref{signedspaces}).  Let $v \in \HH^N$.

If $P_{V_-} v \neq 0$, then $rank\ \begin{bmatrix} T & v \\ v^* & 1 \end{bmatrix} > rank\ T$.
\end{lemma}
\begin{proof}
  Since $P_{V_-} v \neq 0$, there must be some $e_i$, $\lambda_i < 0$, such that $c_i = \langle v, e_i \rangle \neq 0$.  Let us first consider only the vector $c_i e_i$.  $c_i e_i$ is in the range of $T$;  its preimage is $\{ (c_i/\lambda_i) e_i + \nu : \nu \in Null\ T\}$.  So we have that $rank\ \begin{bmatrix} T \\ (c_i e_i)^* \end{bmatrix} = rank\ T$.\\  
        
        We know $\begin{bmatrix} c_i e_i \\ 1 \end{bmatrix}$ is in the range of $\begin{bmatrix} T \\ (c_i e_i)^* \end{bmatrix}$.  We proceed by contradiction.  We know that any solution $w$ for the following equation:
        
$$ \begin{bmatrix} T \\ (c_i e_i)^* \end{bmatrix} (w) = \begin{bmatrix} T (w) \\ (c_i e_i)^* \, w \end{bmatrix} = \begin{bmatrix} c_i e_i\\ \langle w, c_i e_i \rangle \end{bmatrix}$$

is of the form $w = (c_i/\lambda_i) e_i + \nu$ for some $\nu \in Null\ T$.  Yet, we see that in the $N+1^{th}$ slot in the above vector, we have $\langle (c_i/\lambda_i) e_i + \nu, c_i e_i \rangle = |c_i|^2 / \lambda_i = 1$, i.e., $|c_i|^2 = \lambda_i < 0$.  This is a contradiction.

This will suffice to show that for any eigenvector $e_i$ with negative eigenvalue, if we let $P_i$ be the one dimensional projection onto this vector and if $P_i v \neq 0$, then $rank\ \begin{bmatrix} T & P_i v \\ (P_i v)^* & 1 \end{bmatrix} > rank\ T$.  It follows that if $P_{V_-} v \neq 0$, then $rank\ \begin{bmatrix} T & P_{V_-} v \\ (P_{V_-} v)^* & 1 \end{bmatrix} > rank\ T$.  We extend this idea:

The preimage of $P_{V_-} v = \sum_{i \in I_-} c_i e_i$ with regards to $T$ is 
$T^{-1} (P_{V_-} v) = \{ \sum_{i \in I_-} \frac{c_i}{\lambda_i} e_i + \nu \ : \ \nu \in Null\ T \}$.  Thus, for $\begin{bmatrix} P_{V_-} v \\ 1 \end{bmatrix}$ to be in the range of $\begin{bmatrix} T \\ (P_{V_-} v)^* \end{bmatrix}$, we must have 
\begin{align*}
\begin{bmatrix} T \\ (P_{V_-} v)^* \end{bmatrix} ( \sum_{i \in I_-} \frac{c_i}{\lambda_i} e_i + \nu )& = \begin{bmatrix} T(\sum_{i \in I_-} \frac{c_i}{\lambda_i} e_i + \nu )  \\ (P_{V_-} v)^* (\sum_{i \in I_-} \frac{c_i}{\lambda_i} e_i + \nu ) \end{bmatrix}\\
&= \begin{bmatrix} \sum_{i \in I_-} c_i e_i \\ \sum_{i \in I_-} |c_i|^2 / \lambda_i  \end{bmatrix} = \begin{bmatrix} P_{V_-} v \\ 1 \end{bmatrix}
\end{align*}
but this would mean that $|c_i|^2 / \lambda_i = 1$, when $|c_i|^2 / \lambda_i$ is a negative number.

\end{proof}

\begin{lemma}
\label{coefflemma}
With the notation above, 
let $\{a_i\}_{i \in I_+}$ where $I_+$ is the index of eigenvectors with positive eigenvalues.  Let 

\begin{equation}\label{aform} v = \sum_{i \in I_+} a_i \sqrt{\lambda_i} e_i.\end{equation}  Assume that

$$rank\ \begin{bmatrix} T & v \\ v^* & 1 \end{bmatrix} = rank\ T$$

then $\sum_{i \in I_+} |a_i|^2 = 1.$
\end{lemma}
\begin{proof}

Let $v$ be of the form as in (\ref{aform}).  We assume that $rank\ \begin{bmatrix} T & v \\ v^* & 1 \end{bmatrix} = rank\ T$.\\

Then $v$ is in the range of $T$, so $\begin{bmatrix} T \\ v^* \end{bmatrix}$ is of the same rank as $T$.  The preimage of $v$ is $T^{-1} (v) = \{ \sum_{i \in I_+} \frac{a_i}{\sqrt{\lambda_i}} e_i + \nu \ :\ \nu \in Null\ T\}$.  If we let $\nu$ be arbitrary, then 

\begin{align*}
\begin{bmatrix} T \\ v^* \end{bmatrix} ( \sum_{i \in I_+} \frac{a_i}{\sqrt{\lambda_i}} e_i + \nu)& = \begin{bmatrix} T( \sum_{i \in I_+} \frac{a_i}{\sqrt{\lambda_i}} e_i + \nu)  \\ \langle ( \sum_{i \in I_+} \frac{a_i}{\sqrt{\lambda_i}} e_i + \nu), ( \sum_{i \in I_+} {a_i}{\sqrt{\lambda_i}} e_i ) \rangle \end{bmatrix}\\& = \begin{bmatrix} v \\ \sum_{i \in I_+} |a_i|^2 \end{bmatrix}
\end{align*}

This will force $\sum_{i \in I_+} |a_i|^2 = 1$.
\end{proof}

\begin{corollary}
Let $e_i$ be an eigenvector with positive eigenvalue.  Then $rank\ \begin{bmatrix} T & c e_i \\ (c e_i) ^* & 1 \end{bmatrix} = rank\ T$
if and only if $|c| = \sqrt{\lambda_i}$.
\end{corollary}
\begin{proof}
(``$\Leftarrow$")  This is shown in the prior section.\\\\
(``$\Rightarrow$")  Apply lemma (\ref{coefflemma}) with $a_i = 1$, and $a_j = 0, \ $ for $j \neq i$.
\end{proof}

\begin{proposition}
\label{backwardsprop}
Let $v$ be a vector such that $rank\ \begin{bmatrix} T & v \\ v^* & 1 \end{bmatrix} = rank\ T$ for a positive semi-definite matrix $T$.  Let $\{e_i\}_{i=1}^N$ be the eigenvectors for $T$ with associated eigenvalues $\{ \lambda_i \}_{i=1}^N$.  We use $I_+ \subset \{1, \ldots, N\}$ as the index of the positive eigenvalues, i.e., $\lambda_i > 0 \Leftrightarrow i \in I_+$.\\

Let $v \in \HH^N$.  If

$$rank\ \begin{bmatrix} T & v \\ v^* & 1 \end{bmatrix} = rank\ T$$

then $v \in span_{i \in I_+} e_i$, where $v = \sum_{i \in I_+} a_i \sqrt{\lambda_i} e_i$ for some collection of scalars $\{a_i\}_{i \in I_+}$ such that $\sum_{i \in I_+} | a_i |^2 = 1$.
\end{proposition}
\begin{proof}
We start with the assumption $rank\ \begin{bmatrix} T & v \\ v^* & 1 \end{bmatrix} = rank\ T$.  By lemmas (\ref{nulllemma}) and (\ref{negativelemma}), we have $v \in span_{i \in I_+} e_i$.  By lemma (\ref{coefflemma}), we have the conclusion.
\end{proof}

\subsection{Proof of theorem \ref{samerankfamily1}}
\begin{proof}
(``$\Rightarrow$")  This is shown by proposition (\ref{forwardsprop}).\\\\
(``$\Leftarrow$")  This is shown by proposition (\ref{backwardsprop}).\\
\end{proof}


\begin{appendix}

\section*{Appendix: Alternative Proof of Corollary~\ref{riesz_riesz}}\label{brians_argument}

In this section we give a direct proof of Corollary~\ref{riesz_riesz}. Before proving the main result of this section, we need a computational lemma.

\begin{lemma*}
Given operators $S=(b_{ij})_{i,j=1}^N$ and  $T=(a_{ij})_{i,j=1}^N$ on $\HH^N$ we have
\[ \langle T,S \rangle_F = \sum_{i,j=1}^Na_{ij}b_{ij},\]
Moreover, 
\[ \|S\|^2_F = \sum_{i,j=1}^N a_{ij}^2 = \sum_{i=1}^N \|R_i\|^2 = \sum_{i=1}^N\|C_i\|^2,\]
where $R_i$ (resp. $C_i$) is the $i^{th}$-row vector of $S$ (resp. $i^{th}$-column vector of $S$).
\end{lemma*}

\begin{proof}
Note that 
\begin{eqnarray*}\mathrm{Tr} ( S^*T)&=& \mathrm{Tr}\left(
\begin{bmatrix} 
b_{11}&b_{21}& \cdots & b_{N1}\\
b_{12}&b_{22}&\cdots & b_{N2}\\
\vdots&\vdots&\cdots&\vdots\\
b_{1N}&b_{2N}&\cdots & b_{NN}
\end{bmatrix}
\begin{bmatrix}
a_{11}&a_{12}&\cdots &a_{1N}\\
a_{21}&a_{22}&\cdots &a_{2N}\\
\vdots & \vdots & \cdots & \vdots\\
a_{N1}&a_{N2}&\cdots &a_{NN}
\end{bmatrix}\right)\\
&=&\mathrm{Tr}
\begin{bmatrix}
\sum_{i=1}^Nb_{i1}a_{i1}& * & \cdots & *\\
*& \sum_{i=1}^N b_{i2}a_{i2} & \cdots &*\\
\vdots & \vdots & \ddots &  \vdots\\
* & * & \cdots & \sum_{=1}^Nb_{iN}a_{iN}
\end{bmatrix}\\
&=& \sum_{i,j=1}^N a_{ij}b_{ij}.
\end{eqnarray*}

For the moreover part, we have $S^*S = (a_{ji})(a_{ij})$ has diagonal elements $\sum_{j=1}^Na_{ij}^2$ for $i=1,2,\ldots,N$.
\end{proof}

\begin{theorem*}\label{unit_norm_riesz}
Let $\{\phi_i\}_{i=1}^N$ be a unit norm Riesz sequence
 in $\HH^N$ with Riesz bounds $A,B$.  Then $\{\phi_i\phi_i^*\}_{i=1}^N$ has Riesz bounds $A,B$.
\end{theorem*}

\begin{proof}
Given scalars $(a_i)_{i=1}^N$, we have that the $(i,j)$-entry of 
\[S=\sum_{i=1}^Na_i\phi_i\phi_i^*,\]
is
\[ \sum_{k=1}^Nc_k\phi_k(i)\phi_k(j).\] 
So the $i^{th}$-row vector is
\[ R_i=\sum_{k=1}^Nc_k\phi_k(i)\phi_k.\]
So  by our lemmas,
\begin{eqnarray*}
\|S\|^2 &=&\sum_{i=1}^N \|R_i\|^2\\
&=&\sum_{i=1}^N\| \sum_{k=1}^Nc_k\phi_k(i)\phi_k\|^2\\
&\ge& A\sum_{i=1}^N\sum_{k=1}^N|c_k|^2|\phi_k(i)|^2\\
&=&A \sum_{k=1}^N |c_k|^2 \sum_{i=1}^N|\phi_k(i)|^2\\
&=&A \sum_{k=1}^N|c_k|^2.
\end{eqnarray*}
The upper bound is done similarly.
\end{proof}

{\noindent \textbf{THANKS:}}
The authors wish to thank Janet C. Tremain for her helpful comments on this paper.

\end{appendix}
\bibliographystyle{plain}
\bibliography{Riesz_OP_arxiv}

\end{document}